\documentclass[12pt]{article}

\usepackage{geometry}
\usepackage{graphicx}
\usepackage{amsmath,amssymb,amsthm}
\usepackage{enumerate}
\usepackage{hyperref}
\usepackage[T1]{fontenc}
\usepackage{color}
\newtheorem{theo}{Theorem}

\newtheorem{lem}{Lemma}[section]
\newtheorem{defi}{Definition}[section]

\newtheorem{rmk}{Remark}[section]
\newcommand{\eps}{\varepsilon}
\newcommand{\R}{\mathbb{R}}

\newcommand{\Mm}{\mathcal{M}^+}
\renewcommand{\o}{\overline}

\newcommand{\p}{\partial}
\newcommand{\E}{\mathcal{E}}
\numberwithin{equation}{section}
\DeclareMathOperator{\dive}{div}
\DeclareMathOperator{\grad}{grad}
\def\be{\begin{equation}}
\def\ee{\end{equation}}

\newcommand{\rd}{\mathrm{d}}
\newcommand{\ub}{\mathbf{u}}
\newcommand{\fb}{\mathbf{f}}
\newcommand{\mb}{\mathbf{m}}

\renewcommand{\d}{\mathtt{d}}
\newcommand{\D}{\mathtt{D}}
\newcommand{\keywords}[1]{\par\noindent
{\small{\em Keywords\/}: #1}}
\begin{document}
\title{A fitness-driven cross-diffusion system from polulation dynamics as a gradient flow}
\date{}
\author{Stanislav Kondratyev, L\'{e}onard Monsaingeon and Dmitry Vorotnikov}

\maketitle
\begin{abstract}
We consider a fitness-driven model of dispersal of $N$ interacting populations, which was previously studied merely in the case $N=1$.
Based on some optimal transport distance recently introduced, we identify the model as a gradient flow in the metric space of Radon measures.
We prove existence of global non-negative weak solutions to the corresponding system of parabolic PDEs, which involves degenerate cross-diffusion.
Under some additional hypotheses and using a new multicomponent Poincar\'e-Beckner functional inequality, we show that the solutions converge exponentially to an ideal free distribution in the long time regime.
\end{abstract}
\tableofcontents

\vspace{10pt}


\vspace{5pt}


\section{Introduction}

Living organisms
tend to form distributional patterns but not to be 
arranged either uniformly or randomly. 
This spatial heterogeneity plays a crucial role in 
ecological theories and their practical applications. It should be taken into account when modeling 
epidemics, ecological catastrophes, competition, adaptation, 
maintenance of species diversity, parasitism, population growth and decline, social behaviour, and so on \cite{LF89}. In order to understand the way the species distribute themselves it is important to pay attention to their dispersal strategies.  

In this paper we study a system of PDEs for several interacting populations whose dispersal strategy is determined by a local intrinsic characteristic of organisms called \emph{fitness} (cf. \cite{cosner05,cos13}), essentially the signed difference between available resources and their consumption by the individuals. 
The fitness manifests itself as a growth rate, and simultaneously affects the dispersal as the species move along its gradient towards the most favorable environment. The equilibrium when the fitnesses of all populations vanish can be referred to as the \emph{ideal free distribution} \cite{FC69,fr72}, since no net movement of individuals occurs in this case. 
We are thus going to study the system
\begin{equation}
\label{eq:zveriint}
\p_t u_i = -\operatorname{div} (u_i \nabla f_i) + u_i f_i,\qquad x \in \Omega, t 
> 0,\qquad\ i=1,\dots,N,
\end{equation}
of $N$ interacting species located in a bounded domain $\Omega\subset\R^d$.
For prescribed resources $\mb=(m_i(x))$ we assume a generic linear relation between the population densities $\ub=(u_i(t,x))$ and their corresponding fitnesses $\fb=(f_i(t,x))$:
\begin{equation}
 \label{e:fbint}
 \fb=\mb-A \ub.
 \end{equation}
 We assume that both the matrix $A$ and the vector $\mb$ depend on $x\in\Omega$, thus our model is spatially heterogeneous.
 Formula \eqref{e:fbint} expresses the idea that the fitness is determined by the difference between the available resources $\mb$ and the animals' consumption $A\ub$.
 
 The mathematical difficulties which we will face when studying the parabolic system \eqref{eq:zveriint}-\eqref{e:fbint} come from the fact that it involves both cross-diffusion (for $N>1$) and degenerate diffusion.   
In the case of merely one population ($N=1$), the fitness-driven dispersal model 
\eqref{eq:zveriint},  \eqref{e:fbint} was suggested in \cite{mc90,cosner05} and studied in \cite{CW13,KMV15} (see also \cite{ccl08}).
Related fitness-driven two-species models were investigated in \cite{ccl13,ltw14} where one population uses the fitness-driven dispersal strategy and the other diffuses freely or does not move at all. In the case when $A$ is a constant matrix, $\mb\equiv 0$, and the second (reaction) term $u_i f_i$ in \eqref{eq:zveriint} is omitted,  system \eqref{eq:zveriint},  \eqref{e:fbint} is equivalent to the degenerate cross diffusion system which was recently analyzed in \cite{Mon14} with an application to seawater intrusion. 
Another population dynamics model which involves cross-diffusion is the Shigesada, Kawasaki and Teramoto model \begin{equation}
\label{eq:akt}
\p_t u_i = \Delta \left(u_i\left(d_i+\sum\limits_{i=1}^N a_{ij} u_j\right)\right) + u_i \left(\left(c_i-\sum\limits_{i=1}^N b_{ij} u_j\right)\right),\ i=1,\dots,N,
\end{equation} where the coefficients are non-negative constants. It has been extensively studied (mostly for $N=2$) from the point of view of well-posedness and long-time behaviour (see, e.g., \cite{CLY04,CJ06,J15,JZ15,LW15,LN96} and the references therein). Note that the constants $d_i$ in \eqref{eq:akt} are usually assumed to be strictly positive, hence this problem is not as degenerate as our system \eqref{eq:zveriint},  \eqref{e:fbint}.

On the other hand, being inspired by the ideas of the Monge-Kantorovich optimal transport theory \cite{villani03topics,villani08oldnew}, we have recently constructed in \cite{KMV15} a new distance on the space of non-negative finite Radon measures.
The same distance was almost simultaneously introduced in the independent works \cite{peyre_1_2015,LMS_big_2015} (see also subsequent \cite{LMS_small,peyre_2,MG16}).
This metric generates a formal (infinite dimensional) Riemannian structure on this space, and provides first- and second-order differential calculus in the spirit of Otto \cite{otto01}.
With this differential calculus at hand we were able to identify in \cite{KMV15} the scalar model as a metric gradient flow, which allowed us to prove long-time convergence to the ideal free distribution with explicit exponential rates.
The goal of this paper is to extend our previous results to the multispecies case $N>1$: we will observe that the problem 
\eqref{eq:zveriint},  \eqref{e:fbint} can be interpreted as a formal gradient flow of some driving entropy functional on the Cartesian product of $N$ spaces of non-negative Radon measures with respect to this geometric structure (provided that the matrix $A(x)$ is symmetric).
Roughly speaking, the entropy $\E(\ub)\approx \|\ub -\ub^\infty\|^2_{L^2(\Omega)}\approx\|\fb\|^2_{L^2(\Omega)}$ will quantify the deviation from the ideal free distribution $\ub^\infty$, characterized by $\fb=0$.
In this framework and under some general structural assumptions on $A$ and $\mb$, we will prove existence of non-negative weak solutions to our problem (which to the best of our knowledge was known so far only in the scalar case \cite{CW13}), and show that, at least for subcritical initial entropies,
all the species persist and exponentially converge to the ideal free distribution.
All our arguments will have a strong optimal transport flavor, but, due to the multicomponent nature of the problem preventing our entropy functional from being geodesically convex, the abstract results for metric gradient flows in \cite{AGS06} do not apply directly.
As a consequence some technical work will be needed to justify the formal Riemannian computations and a priori estimates (in particular some chain rules in weighted spaces), and we will argue using several approximations and regularizations in a more standard PDE framework.
\\

The paper is organized as follows: In Section~\ref{section:animals} we impose precise structural assumptions and state our results.
In Section~\ref{section:gradient_flow} we describe the optimal transport distance on the product measure space $\mathcal M^+(\Omega)^N$, discuss the induced formal Riemannian structure and calculus, and highlight the gradient-flow structure of the system.
We also present two formal but crucial computations, consisting of a priori entropy and gradient estimates to be derived more rigorously later on.
Section~\ref{section:existence} is devoted to the existence of weak solutions, whose proof will involve three successive regularizations.
The last Section~\ref{section:long_time_convergence} contains the proof of the long-time convergence, based on a new vectorial Poincar\'e-Beckner type inequality which serves as an entropy-entropy production inequality for our gradient flow. The Appendix contains a technical functional analytic lemma.

\section{Conventions and main results} \label{section:animals}
Throughout the whole paper we assume that $\Omega\subset \R^d$ is an open, connected, bounded domain with sufficiently smooth boundary.
We always denote vector functions with values in $\R^N
$ by bold letters, e.g., $\fb=(f_i(t,x))$.
We assume that we are given a function $\mb=(m_i(x)):\overline\Omega\to \R^N$ and a symmetric positive-definite matrix function $A=(a_{ij}(x)):\overline\Omega\to \R^{N\times N}$, and without further mention we shall always assume the uniform ellipticity condition
$$
0<\lambda_A \leq A(x)\leq \Lambda_A,\qquad x\in \overline\Omega
$$
for some structural constants $\lambda_A\leq \Lambda_A$.
We assume that $A$ and $\mb$ are sufficiently smooth. Note that we do not assume that all the components of $A$ and $\mb$ are non-negative. 

All the integrals are always implicitly written with respect to the Lebesgue measures $\rd x,\rd t,$ or $\rd x\rd t$.  The symbol $\Mm(\Omega)$ denotes the space of (non-negative) 
Radon measures on $\Omega$. Parentheses denote the scalar product in $L^2(\Omega)$, $L^2(\Omega)^N$ or $L^2(\Omega)^{N\times d}$. The symbol $\mathcal{C}_w(\mathcal J; X)$ stands for the space of weakly continuous (resp., narrowly continuous) curves with defined on $\mathcal J\subset \R$ and with values in $X=L^2(\Omega)$ (resp., $X=\Mm(\Omega)$).


We study the system
\begin{equation}
\label{eq:zveri}
\left\{
\begin{array}{ll}
\displaystyle
\p_t u_i = -\operatorname{div} (u_i \nabla f_i) + u_i f_i, & x \in \Omega,\, t 
> 0,\ i=1,\dots,N, 
\\
u_i
\frac{
\partial f_i
}{
\partial \nu
}
=
0
,
&
x \in \partial \Omega, \ t > 0,
\\
u_i(0, x) = u_{i}^0(x), & x \in \Omega
.
\end{array}
\right.
\end{equation}
As already mentioned in the introduction and without further mention, we always denote the fitness by
\begin{equation*}
 \label{e:fb}
 \fb=\mb-A \ub,
 \end{equation*}
 and the ideal free distribution $\ub^\infty(x)$ is obtained by solving $\fb =0$, i-e
 $$
 \ub^{\infty}:=A^{-1}\mb\qquad \Leftrightarrow \qquad \fb=0.
 $$
 Note that at this stage $\ub^\infty$ can have negative components and may therefore be biologically irrelevant (but it will be non-negative later on with extra structural conditions on $A,\mb$), and that $\ub^{\infty}$ is trivially a steady state of \eqref{eq:zveri} with $f_i\equiv 0$.

We will show in Section~\ref{section:gradient_flow} that \eqref{eq:zveri} is the gradient flow of the entropy 
\begin{multline} \label{e:15}
\E(\mathbf u)=\frac 12 \int_\Omega A(\ub-\ub^{\infty})\cdot(\ub-\ub^{\infty})\\ = \frac 1 2\int_\Omega A^{-1} \fb \cdot \fb=\frac 12 \int_\Omega (\ub^{\infty}-\ub)\cdot\fb, \quad \ub\in L^2(\Omega)^N
\end{multline}
with respect to some optimal transport distance.

\begin{defi}
Let $\mathbf u^{0}\in L^2(\Omega)^N$, $u^0_{i}\geq 0$.
A vector function
$$
\ub\in \mathcal{C}_w([0,\infty); L^2(\Omega)^N)\cap L^2_{loc}([0,+\infty); H^1(\Omega)^N),
$$
$u_i\geq 0$, is called a \emph{non-negative weak solution} to problem \eqref{eq:zveri}  provided
\begin{equation*}
\frac d {dt} \int_\Omega  \ub \cdot \mathbf w= \sum_{i=1}^N  \int_\Omega \left( u_i \nabla f_i\cdot \nabla w_i +  u_i  f_i w_i\right), \quad \forall \mathbf w \in (\mathcal{C}^1(\o\Omega))^N, \label{eq:wsol}
\end{equation*}
in the sense of distributions $\mathcal D'(0,\infty)$, and the initial condition $\ub(0)=\ub^0$ is satisfied weakly in the space $L^2(\Omega)^N$.
\end{defi}

\begin{theo}[existence of non-negative weak solutions] \label{T:ex1}
Let $\mathbf u^{0}\in L^2(\Omega)^N$ with $u^0_{i}\geq 0$.
There exists a non-negative weak solution
\begin{equation*} 
\label{eq:solclass}
\mathbf{u}\in L^\infty(0,+\infty; L^2(\Omega)^N)\bigcap \mathcal{C}_w([0,+\infty); L^2(\Omega)^N)
\bigcap L^{2}_{loc}([0,+\infty); H^1(\Omega)^N)
\end{equation*}
 to problem \eqref{eq:zveri}, satisfying the \emph{Entropy-Dissipation-Inequality}
\begin{equation}
\label{eq:EDI}
\E(\ub(t_1))+\sum\limits_i\int_{t_0}^{t_1}\int_{\Omega} u_i(|\nabla f_i|^2+|f_i|^2)\leq \E(\ub(t_0))
\quad
\mbox{for a.e. }0\leq t_0\leq t_1.
\end{equation}
\end{theo}
\begin{rmk} We were not able to prove any uniqueness result due to the lack of geodesic convexity, which usually gives contractivity in the metric sense and thus uniqueness; we therefore believe that any hypothetical proof of uniqueness cannot come from standard mass transport arguments and should rely on some PDE approach. \end{rmk}

To fix the ideas and improve readability, we assume that $A$ does not depend on $x$.
The generalization to the $x$-dependent case would require technical work and employs the fact that the last sum in the expansion
$
 \nabla f_i=\nabla m_i-\sum_{j}a_{ij}\nabla u_j -\sum_{j}u_j\nabla a_{ij}
 $
 is of lower order with respect to the penultimate one, but all the arguments below would carry through with minor modifications.

With an additional assumption, we obtain long-time convergence $\ub(t)\to \ub^{\infty}$ to the ideal free distribution with survival of all the species:
\begin{theo}[long-time behavior] \label{T:lt1}
Let $\ub^0$ be as in Theorem \ref{T:ex1}, and assume in addition that the structural hypothesis \eqref{eq:am4} holds.
Then $\ub^{\infty}(x)> 0$ (componentwise), and there exists $E^*\equiv E^*(A,\mb)>0$ such that, for any subcritical initial datum $\E(\ub^0)<E^*$, our solution $\ub$ converges exponentially to $\mathbf u^\infty$ as
\begin{equation}
\label{eq:din}
\E(\mathbf u(t))\leq e^{-\gamma t} \E(\mathbf u^{0})
\end{equation}
for all $t\geq 0$ and some $\gamma\equiv \gamma(\ub^0,\mb,A,\Omega)>0$.
\end{theo}

Note that our coercivity assumption $A\geq \lambda_A$ controls $\E(\ub)\geq \frac{\lambda_A}{2}\|\ub-\ub^\infty\|^2_{L^2(\Omega)}$, thus the entropy decay \eqref{eq:din} immediately implies $L^2$ convergence.
The consequences and interpretation of the additional assumption \eqref{eq:am4} will be discussed later on in Section~\ref{section:long_time_convergence}.
In Theorem~\ref{T:lt1} we had to restrict to subcritical entropies $\E(\ub^0)<E^*$ for technical reasons, but we conjecture that \eqref{eq:din} holds for any $\ub^0\geq 0$ (unless some component of $\ub^0$ is identically zero).
Indeed our proof of the long-time convergence works provided that no extinction occurs, say $\|u_i(t)\|_{L^p(\Omega)}\geq c_i>0$ for all $i\in \{1,\ldots,N\}$, $t\geq 0$, and some $p\geq 1$.
For both the ODE dynamics (i-e when all densities and resources are constant in space) and for the one-animal PDE dynamics \cite{CW13,KMV15} this is true, but due to the cross-diffusion we were not able to prove the non-extinction in the general case.
In other words, our subcriticality assumption in Theorem~\ref{T:lt1} is a technical workaround guaranteeing that our solution stays away from a finite number of certain partial extinction regimes.
These regimes correspond to the situations when some (or all) populations go extinct, and the survivors compose a (lower-dimensional) ideal free distribution.
This allows us to provide an explicit value for $E^*$ depending only on the structure of the problem, see Section~\ref{section:long_time_convergence} for the details.

\section{The gradient-flow structure and a priori estimates}
\label{section:gradient_flow}
The celebrated Benamou-Brenier formula was originally established in \cite{BenamouBrenier00} to characterize the quadratic Monge-Kantorovich-Wasserstein distance as a dynamical evolution problem, and is restricted to \emph{conservative} optimal transport of measures with fixed mass (typically probability measures). 
In \cite{KMV15} we constructed an optimal transport distance $\mathtt d$ on the space of \emph{arbitrary} non-negative Radon measures $\Mm(\Omega)$, based on a modified dynamical Benamou-Brenier formula allowing for mass variations.
More precisely, for any $u_0,u_1\in \Mm(\Omega)$ the distance reads
$$
\d^2(u_0,u_1)=\min\limits_{(u,g)}\int_0^1\int_\Omega(|\nabla g_t(x)|^2+|g_t(x)|^2)\rd u_t(x)\rd t,
$$
where the infimum is taken among narrowly continuous curves $$t\mapsto u_t\in \mathcal C_w([0,1];\Mm(\Omega))$$ with endpoints $u_0,u_1$ such that
the non conservative continuity equation
$$
\partial_t u_t+\dive(u_t\nabla g_t)=u_t g_t
$$
holds in the sense of distributions $\mathcal D'((0,1)\times\Omega)$.
Biologically, one can think of $g$ as fitness: in the continuity equation above the individuals $u$ reproduce or die with rate $g$ equal to the local fitness, and move along the velocity field $\nabla g$ towards the most favorable environment.
Our construction was originally derived in the whole space $\Omega=\R^d$, but immediately extends to general domains imposing natural zero-flux boundary conditions on the velocity fields $\nabla g_t$ on $\partial\Omega$.
In addition to nice geometrical and topological properties (completeness, existence of geodesics, lower semi-continuity with respect to $\text{weak-}*$ convergence, characterization of Lipschitz curves...) the metric $\mathtt d$ gives a formal Riemannian structure \emph{\`a la Otto} \cite{otto01} on the space $\Mm(\Omega)$, endowing the tangent plane
$$
T_u\Mm=\{\partial_t u=\zeta:\qquad \zeta=-\dive(u\nabla g)+u g,\quad g\in H^1(\rd u)\}
$$
with the norm
\begin{equation}
\label{eq:def_tangent_norm}
\|\zeta\|^2_{T_u\Mm}=\|g\|^2_{H^1(\rd u)}=\int_\Omega(|\nabla g|^2+|g|^2)\rd u
\end{equation}
and scalar product
$$
\left<\zeta_1,\zeta_2\right>_{T_u\Mm}:=\left(g_1,g_2\right)_{H^1(\rd u)}=\int_\Omega(\nabla g_1\cdot\nabla g_2 + g_1 g_2)\rd u.
$$
Here tangent vectors $\partial_t u=\zeta\in T_u\Mm$ are always identified with scalar potentials $g\in H^1(\rd u)$ through the elliptic equation $-\dive(u\nabla g)+ug=\zeta$ (supplemented with zero-flux boundary conditions on $\partial\Omega$ if needed).
In particular this allows to compute metric gradients for functionals $\mathcal F(u)=\int_\Omega F(x,u)$ on $\Mm$ as
\begin{equation}
\label{eq:formula_grad_d_scalar}
\grad_\d \mathcal F(u)=-\dive\left(u\nabla \frac{\delta F}{\delta u}\right)+u\frac{\delta F}{\delta u},
\end{equation}
where $\frac{\delta F}{\delta u}=\partial_uF(x,u)$ stands for the first variation with respect to $u$ and $\nabla=\nabla_x$ is the usual gradient in space.
We refer to \cite{KMV15} for further details and explanations.

Since we want to deal here with multicomponent variables $\ub=(u_1,\ldots,u_N)$ we endow $(\Mm(\Omega))^N$ with the natural product distance
$$
\D^2(\ub,\mathbf{v})=\sum\limits_{i=1}^N\d^2(u_i,v_i),
$$
giving the natural Riemannian metrics
$$
\left<\boldsymbol{\zeta}^1,\boldsymbol\zeta^2\right>_{T_\ub(\Mm)^N}=\left(\mathbf g^1,\mathbf g ^2\right)_{H^1(\rd\ub)}=\sum\limits_{i=1}^N\int_\Omega (\nabla g_i^1\cdot \nabla g_i^2+g_i^1g_i^2)\rd u_i
$$
in the tangent space $T_\ub(\Mm)^N=\overset{N}{\underset{i=1}{\oplus}} T_{u_i}\Mm$.
Here we identify again the tangent vectors $\boldsymbol \zeta=(\zeta_1,\ldots,\zeta_N)$ with potentials $\boldsymbol g=(g_1,\ldots g_N)$ via the elliptic equations $-\dive(u_i\nabla g_i)+u_ig_i=\zeta_i$ (with homogeneous Neumann boundary conditions).
The metric derivatives with respect to $\D$ can be simply computed applying \eqref{eq:formula_grad_d_scalar} component by component, i-e gradients of functionals $\mathcal F(\ub)=\int_\Omega F(x,u_1,\ldots,u_N)$ read
$$
\grad_\D\mathcal F(\ub)=\left(-\dive\left(u_i\nabla\frac {\delta F}{\delta u_i}\right)+u_i\frac {\delta F}{\delta u_i}\right)_{i=1\ldots N}.
$$
For the particular case $\E(\ub)=\frac{1}{2}\int_\Omega A(\ub-\ub^{\infty})\cdot (\ub-\ub^{\infty})$ and with the previous notation $\fb=\mb-A\ub$, we have $\frac{\delta F}{\delta u_i}=(A(\ub-\ub^\infty))_i=(A\ub-\mb)_i=-f_i$, thus our system of PDEs can indeed be written as the gradient flow
\begin{equation}
\label{eq:structure_PDE_gradient_flow}
\frac{d\ub}{dt}=-\grad_{\D} \E(\ub)
\quad \Leftrightarrow \quad
\partial_t u_i=-\dive(u_i\nabla f_i)+u_if_i.
\end{equation}
\\

At this stage let us derive two formal estimates, which will be crucial for the subsequent analysis.
Here we ignore all the regularity issues, and we shall make these estimates rigorous throughout the several regularized problems involved in the proof of existence.
The first estimate is the Entropy-Dissipation-Inequality \eqref{eq:EDI}, and is inherent to the gradient-flow structure.
Indeed from \eqref{eq:structure_PDE_gradient_flow} we should have along reasonably smooth solutions that
$$
 \frac{d \E}{dt}=\left<\grad_\D \E (\ub),\frac{d\ub}{dt}\right>_{T_\ub(\Mm)^N}=-\|\grad_\D \E (\ub)\|^2_{T_\ub(\Mm)^N}.
$$
Given the above definition of the tangent norms and the explicit computation of the metric gradient in terms of the fitness, this reads in our setting
$$
\E(\ub(t_1))+\sum\limits_i\int_{t_0}^{t_1}\int_{\Omega} u_i(|\nabla f_i|^2+|f_i|^2) = \E(\ub(t_0))
\quad
\mbox{for all }0\leq t_0\leq t_1.
$$
This is often referred to as the Entropy-Dissipation-\emph{Equality} for the obvious reasons, and implies of course \eqref{eq:EDI}.
However the latter \emph{inequality} is well known to still completely characterize metric gradient flows \cite{AGS06}, and will turn out to be more flexible and easier to obtain rigorously along the various approximations.

The second fundamental estimate, which will serve as a technical tool, is the a priori gradient estimate
$$
\|\nabla \ub\|^2_{L^2(0,T;L^2(\Omega)^{N\times d})}\leq C(1+T)
$$
and can be viewed as a ``flow interchange'' estimate as introduced in \cite{MMS_flow_interchange}.
Indeed the estimate formally follows from computing the dissipation of the Boltzmann entropy $\mathcal H(u_i)=\int_\Omega \{u_i\log u_i-u_i+1\}$ along solutions of our PDE, which is the gradient flow of the driving functional $\E$.
More precisely, testing $\log u_i$ in \eqref{eq:zveri} (recall that our weak solutions will be non-negative) we compute
\begin{multline*}
\frac{d}{dt}\mathcal H(u_i)=\int_\Omega \log u_i\,\partial_t u_i =\int_\Omega \log u_i\big\{-\dive(u_i\nabla f_i+u_i f_i)\big\}\\
= \int \nabla u_i\cdot\nabla f_i + \int_\Omega u_i f_i\log u_i.
\end{multline*}
Writing $f_i=m_i-(A\ub)_i$ in the last gradient term, summing over $i$'s, and integrating in time, this can be rearranged as
\begin{multline*}
\sum\limits_i\left\{ \mathcal H(u_i(T))+\int_0^T\int_\Omega \nabla u_i\cdot \nabla (A\ub)_i\right\}\leq \sum\limits_i\Bigg\{ \mathcal H(u_i(0))\\
+\int_0^T\int_\Omega \nabla u_i\cdot \nabla m_i
+\frac{1}{2}\int_0^T\int_\Omega u_i|f_i|^2
+\frac{1}{2}\int_0^T \int_\Omega u_i|\log u_i|^2\Bigg\}.
\end{multline*}
Observe from $\E(\ub)\geq \frac{\lambda_A}{2} \|\ub-\mb\|_{L^2(\Omega)^N}$ and the previous EDI \eqref{eq:EDI} that $\ub(t)$ is bounded in $L^2(\Omega)^N$, thus the subquadratic terms $\mathcal H(u_i(0))$ and $\int_\Omega u_i|\log u_i|^2$ are controlled uniformly in time in the right-hand side.
Exploiting next $\mathcal H(u_i(T))\geq 0$ and the coercivity $A\geq \lambda_A$, a suitable use of Young's inequality finally gives
$$
\frac{\lambda_A}{2}\|\nabla \ub\|_{L^2(0,T;L^2(\Omega)^N)}^2\leq C(1+T)
+
\frac{2}{\lambda_A}\|\nabla\mb\|^2_{L^2(\Omega)^N}T
+\sum\limits_i\frac{1}{2}\int_0^T\int_\Omega u_i|f_i|^2\leq C(1+T).
$$
Here we used \eqref{eq:EDI} to bound the dissipation term $\int_0^\infty\int_\Omega u_i|f_i|^2\leq \E(\ub^0)\leq C$. We also implicitly assumed that $A$ is a constant matrix, otherwise some extra lower terms appear but the gradient estimate is still true.

\section{Existence of weak solutions}
\label{section:existence}
Our construction of weak solutions will involve three levels of approximation, indexed by the regularization parameters $\eps,\delta\to 0$ and $M\to\infty$.
More precisely, we shall consider the regularized problems
\begin{equation}
\label{eq:zveri_a1}
\left\{
\begin{array}{ll}
\displaystyle
\p_t u_i +\eps  (\mathcal  A \mathbf u)_i= -\operatorname{div} (\tilde u_i \nabla f_i) + \tilde u_i  f_i+\delta \Delta u_i,
\\
u_i(0, x) = u^0_{i}(x)
\end{array}
\qquad \text{with }\fb=\mb - A\ub.
\right.
\end{equation}
Here
$$
\tilde u_i=\max(0,\min(M,u_i))
$$
is the truncation between $0$ and $M\gg1$.
The operator $\mathcal  A$ is a suitable elliptic operator of higher order to be precised shortly together with its associated boundary conditions, and will essentially allow to consider the second-order cross diffusion as a compact perturbation of $\eps\mathcal A$ for fixed $\eps>0$.
Note that in the original PDEs the terms $u_i\nabla f_i$ only belong to $L^1(\Omega)$ if $\ub\in H^1(\Omega)^N$, while the truncations $\tilde u_i\nabla u_i\in L^2(\Omega)$ behave much better for fixed $M<\infty$.
Lastly, the $\delta\Delta$ term will help to gain coercivity and control the degenerate cross-diffusion.

We shall first take $\eps\to 0$, then $M\to\infty$, and finally $\delta\to 0$ in this order.
The most delicate limit will be $M\to\infty$, when we loose the regularity $\tilde u_i\nabla f_i\in L^2(\Omega)$ to the more delicate $u_i\nabla f_i\in L^1(\Omega)$.
This step will also require the rigorous justification of the formal computations and chain rules from Section~\ref{section:gradient_flow}, and will be the most involved.
It is worth pointing out that solutions will become non-negative $u_i(t,x)\geq 0$ only after taking $\eps\to 0$, which will then carry through the next limits $M\to\infty,\delta\to 0$.
For convenience we will work in finite time intervals $[0,T]$.
All our estimates will give local-in time control, and we will retrieve in the end a global solution by standard diagonal extraction.
In order to keep the notation light we omit the $\eps,M,\delta$ indexes as often as possible, and throughout the manuscript we will precise the dependence of the solution on the regularizing parameters when needed.\\

Our first step is to prove existence of solutions to \eqref{eq:zveri_a1} for small $\eps,\delta>0$ and large $M<\infty$ in any fixed time interval $[0,T]$.
In order to give a precise meaning to this problem, consider the Hilbert triple (see the Appendix for the abstract definition)
$$
(H^r(\Omega)^N, L^2(\Omega)^N, {(H^r)}^*(\Omega)^N)
\qquad 
\text{for fixed }r>1+\frac d 2.
$$
Denote by $\mathcal  A$ the Riesz isometry between the spaces $(H^r)^N$ and ${((H^r)}^*)^N$.
We recall the Sobolev embedding $H^r(\Omega)\subset \mathcal C^1(\overline\Omega)$, which is compact.  

The weak form of \eqref{eq:zveri_a1} (with a certain implicit  higher-order Neumann boundary condition which is of no importance to us) is the following Cauchy problem
\begin{equation}
\label{E:apprp1}
\ub' +\eps  \mathcal A\mathbf u=Q(\mathbf u), \qquad \mathbf u|_{t=0}=\mathbf u^0,
\end{equation}  
 where the first equality is understood as an ODE in the space $((H^r)^*)^N$ and the initial datum should be taken in the sense of $\mathcal C([0,T];(L^2)^N$).
The operator $Q:(H^1)^N\to ((H^1)^*)^N$ is determined by duality as
$$
\langle Q(\mathbf u), \mathbf w\rangle= \sum_{i=1}^N  \int_\Omega \left(\tilde u_i \nabla f_i\cdot \nabla w_i + \tilde u_i  f_i w_i-\delta \nabla u_i \cdot \nabla w_i\right), \quad \forall \mathbf w \in (H^1)^N.
$$
Note that, for fixed $M>0$, the operator $Q$ has sublinear growth and is continuous from $\mathcal C^1(\overline\Omega)^N$ to $((H^r)^*)^N$.  By Lemma \ref{l:ap} with $X=H^r(\Omega)^N$, $V=\mathcal C^1(\overline\Omega)^N$, $Y=L^2(\Omega)^N$, there exists a solution $\mathbf u$ to \eqref{E:apprp1} in the class
\begin{equation}
\label{e:skl}
L^{2}(0,T;(H^r)^N)\cap H^{1}(0,T;((H^r)^*)^N)\cap \mathcal C([0,T];(L^2)^N) 
\end{equation}
for any fixed $T>0$.
%
\subsection{The limit $\eps \to 0$}
In this section $M<\infty$ and $\delta>0$ are fixed and only $\eps$ varies.
In order to send $\eps\to 0$ we first rigorously derive an entropy estimate.
To this end, we test \eqref{E:apprp1} with $-\mathbf f$ (this test function is legitimate since $\mathbf f(t,x)=\mb(x)-A\ub(t,x)$ belongs to the intersection \eqref{e:skl}), and arrive at 
\begin{multline}
\label{eq:dissipation_eps}
\frac d {dt} \E(\ub)+\eps  (\ub, A\ub -\mb)_{(H^r)^N}
+\delta(\nabla \ub, \nabla A\ub -\nabla \mb)_{(L^2)^{N\times d}}+ \sum_{i=1}^N  \int_\Omega \tilde u_i (|f_i|^2+|\nabla f_i|^2)=0.
\end{multline}
Integrating in time from $0$ to any $\tau\leq T$, exploiting the coercivity $A\geq \lambda_A\operatorname{Id}$, and applying Cauchy's inequality we find that
\begin{multline}
\label{eq:dissipation_integrated_eps}
\E(\ub(\tau))+\frac 1 2 \eps  \lambda_A\|\ub\|^2_{L^2(0,\tau;(H^r)^N)}+\frac 1 2 \delta \lambda_A\|\nabla \ub\|^2_{L^2(0,\tau;(L^2)^{N\times d})}
+ \sum\limits_i\int_0^\tau \int_\Omega\tilde u_i(|\nabla f_i|^2+|f_i|^2)
\\ \leq \E(\ub^0)+T\frac \eps  {2  \lambda_A}\|\mb\|^2_{(H^r)^N}+T \frac \delta {2  \lambda_A}\|\nabla \mb\|^2_{(L^2)^{N\times d}}.
\end{multline} 
Recalling that $\E(\ub)$ controls the $L^2$ norm, this implies the a priori estimates

\begin{equation}
 \label{l2est}
 \|\ub\|_{L^\infty(0,T;(L^2)^N)}^2\leq C(1+T),
 \end{equation}
\begin{equation}    
\label{delest1}
\|\ub\|^2_{L^2(0,T;(H^1)^N)}\leq C \delta^{-1}(1 +T),
\end{equation}
\begin{equation}
 \label{epsest}
 \|\ub\|^2_{L^2(0,T;(H^r)^N)}\leq C \eps ^{-1} (1 +T)
 \end{equation}
 for small $\eps,\delta$, where the various constants $C$ are independent of $M,\eps,\delta, T$.
Testing \eqref{E:apprp1} by arbitrary $\mathbf w \in (H^r)^N$ and employing \eqref{delest1} and \eqref{epsest}, we deduce 
\begin{equation*}
 \label{tderest1}
 \|\ub^\prime\|_{L^{2}(0,T;((H^r)^*)^N)}\leq C_1,
 \end{equation*}
 where $C_1\equiv C_1(\delta,M,T)$ is independent of $\eps $.
 By the Banach-Alaoglu theorem and Aubin-Lions-Simon lemma, we can find a sequence $\ub^{\eps_k}$ of solutions to \eqref{E:apprp1} with $\eps=\eps_k\to 0$ (for fixed $M,\delta$) such that
 $$
 \ub^{\eps_k}\rightharpoonup \ub\textrm{\ weakly\ in\ }L^2(0,T;(H^1)^N)\textrm{\ and\ weakly-*\ in\ } L^\infty(0,T;(L^2)^N),
 $$
$$
\ub^{\eps_k}\to \ub\textrm{\  strongly\ in\ } L^2(0,T;(L^2)^N)\textrm{\ and\ in\ } \mathcal C([0,T];((H^1)^*)^N),
$$
$$
(\ub^{\eps_k})^\prime\rightharpoonup \ub^\prime\textrm{\ weakly\ in\ }L^{2}(0,T;((H^r)^*)^N).
$$
By the classical continuity property \cite{Kr64} of Nemytskii truncations,
$$
\tilde\ub^{\eps_k}\to \tilde\ub\textrm{\  strongly\ in\ } L^2(0,T;(L^2)^N),
$$
and because $\fb^{\eps_k}=\mathbf m-A\ub^{\eps_k}\to \mathbf m -A\ub=\fb$ strongly in $L^2(0,T;(L^2)^N)$ and weakly in $L^2(0,T;H^1)$ the products
$$
\tilde u_i^{\eps_k}\nabla f_i^{\eps_k}\rightharpoonup \tilde u_i\nabla f_i
\quad\mbox{and}\quad
\tilde u_i^{\eps_k} f^{\eps_k}_i\rightharpoonup \tilde u_i f_i
\qquad
\mbox{weakly in }L^1(0,T;L^1)
$$
converge as strong-weak limits.
Passing to the limit $\eps_k\to 0$ in \eqref{E:apprp1}, we see that $\ub$ solves the problem 
\begin{equation}
\label{E:apprp2} 
\mathbf u^\prime =Q(\mathbf u), \quad \mathbf u|_{t=0}=\mathbf u^0, 
\end{equation}
which is nothing but the weak form of the problem 
\begin{equation}
\label{eq:zveri_apr}
\left\{
\begin{array}{ll}
\displaystyle
\p_t u_i = -\operatorname{div} (\tilde u_i \nabla f_i) + \tilde u_i  f_i+\delta \Delta u_i,&
x \in \Omega,
\\
\tilde u_i
\frac{
\partial f_i
}{
\partial \nu
}-\delta \frac{
\partial u_i
}{
\partial \nu
}
=
0
,
&
x \in \partial \Omega,
\\
u_i(0, x) = u_{0i}(x).
\end{array}
\right.
\end{equation}
By density it is easy to check that $\mathbf u^\prime\in L^{2}(0,T;((H^1)^*)^N)$, and the first equality in \eqref{E:apprp2} holds in the space $(H^1)^*)^N$ for a.e. $t$.
By standard Lions-Magenes interpolation results \cite[Lemma 2.2.7]{ZV08} we have moreover $\mathbf u\in \mathcal C([0,T];(L^2)^N)$. \\

We now show non-negativity of our weak solution $\ub=\lim \ub^{\eps_k}$ to \eqref{eq:zveri_apr} and derive a priori $L^2(0,T;(H^1)^N)$ estimates uniformly in $M,\delta$, which will allow to take the limit $M\to\infty,\delta\to 0$ in the next sections.
After the previous limit $\eps\to 0$ the solution $\ub$ belongs at this stage to $\mathcal C([0,T];(L^2)^N)\cap L^2(0,T;(H^1)^N)$ for fixed $M,\delta>0$.
Therefore we can take again $-\mathbf f$ as a test function and repeat the previous computations \eqref{eq:dissipation_eps}\eqref{eq:dissipation_integrated_eps} with now $\eps=0$, and we get similarly
\begin{equation}
\label{eq:estimate_dissipation_leq_CT}
\E(\ub(t_1))+ \sum\limits_i\int_{t_0}^{t_1}\int_\Omega\tilde u_i(|\nabla f_i|^2+|f_i|^2)
 \leq \E(\ub(t_0))+T \frac \delta {2  \lambda_A}\|\nabla \mb\|^2_{(L^2)^{N\times d}}
\end{equation}
for all $0\leq t_0\leq t_1\leq T$.
This is of course an approximation of the Entropy Dissipation Inequality \eqref{eq:EDI}, which will pass to the successive limits $M\to\infty$ and $\delta\to 0$ later on.

In order show that $u_i\geq 0$ we take the admissible test function $v_i:=\min\{u_i,0\}\in L^2(0,T;H^1)$ in \eqref{eq:zveri_apr}.
By the classical Serrin's chain rule $\nabla v_i=\chi_{[u_i<0]}\nabla u_i$ we get for all components $i=1\ldots N$
$$
 \frac{d}{dt}\left(\frac 12 \int_\Omega |v_i|^2\right)
 = \int_{\Omega}\tilde u_i\nabla f_i\cdot\nabla v_i+\int_\Omega\tilde u_i f_i v_i-\delta\int_\Omega |\nabla v_i|^2\leq 0,
$$
where we used that by definition the truncation $\tilde u_i=\max\{0,\min\{u_i,M\}\}=0$ wherever $v_i=\min\{u_i,0\}\neq 0$ and $\nabla v_i=\nabla u_i \chi_{[u_i<0]}\neq 0$ so that the first two integrands in the middle term are identically zero.
Since we consider non-negative initial data $u_i^0\geq 0$ we have $v_i(0,.)=0$, thus $v_i(t,.)=0$ for all later times and
$$
v_i=\min\{0,u_i\}\equiv 0
\qquad \Rightarrow \qquad u_i(t,x)\geq 0\mbox{ a.e. in }(0,T)\times\Omega.
$$
From now on we slightly abuse the notation and still write $\tilde{u}_i=\min\{u_i,M\}\geq 0$ for the upper truncation only, which is justified since we just proved that $u_i\geq 0$.

In order to mimic the formal gradient estimate from Section~\ref{section:gradient_flow}, we would like to test $\log u_i$ in \eqref{eq:zveri_apr}.
However this is not rigorous because $\log u_i$  may not be an admissible test function.
We use instead the truncated logarithm and Boltzmann entropy, defined as
$$
\log^M_\zeta(z):=\left\{
\begin{array}{ll}
 \log \zeta & \mbox{if }0<z\leq \zeta\\
 \log z & \mbox{if }\zeta \leq z <M\\
 \log M & \mbox{else}
\end{array}
\right.
\quad\mbox{and}\quad
H_\zeta^M(z):=\int_1^z\log_\zeta^M(s)\rd s
$$
for small $\zeta>0$ ($M>0$ is the same truncation level as before).
Note that
$$
0\leq H^M_\zeta(z)\leq H(z):=z\log z-z+1
$$
with monotone pointwise convergence $H^M_\zeta(.)\nearrow H(.)$ as $\zeta\searrow 0$ and $M\nearrow \infty$ (the convergence is actually locally uniform).
We stress at this point that all the next estimates will be uniform in $\delta,M$, and all the constants $C_T$ below will depend on the data and $T>0$ only (if $\delta,\zeta>0$ are small and $M>0$ is large).

For fixed $0<\zeta<M<\infty$ the functions $\log_\zeta^M,H^M_\zeta$ are globally Lipschitz, thus $\log_\zeta^M(u_i)\in L^2(0,T; H^1)$ is an admissible test function in \eqref{eq:zveri_apr} and, using the chain rule $\nabla \log_\zeta^M(u_i)=\frac{\nabla u_i}{u_i}\chi_{[\zeta<u_i<M]}$,
\begin{multline*}
\frac{d}{dt}\left(\int_\Omega H_\zeta^M(u_i) \right) 
=\left<\partial_t u_i,\log_\zeta^M(u_i)\right>_{(H^1)^*,H^1}\\
=\int\limits_{[\zeta<u_i<M]} \nabla f_i\cdot \nabla u_i+\int_\Omega \tilde u_i f_i\log_\zeta^M(u_i)-\delta\int\limits_{[\zeta<u_i<M]}\frac{|\nabla u_i|^2}{|u_i|}.
\end{multline*}
Integrating from $t=0$ to $t=T$, exploiting the monotonicity $0\leq H_\zeta^M\leq H$, applying Young's inequality to $(f_i,\log^M_\zeta u_i)_{L^2(\rd\tilde u_i)}$, and discarding the last non-positive term, we get
\begin{multline*}
0\leq \int_\Omega H^M_\zeta(u_i(T))
\leq  \int_\Omega H^M_\zeta(u_i(0)) + \iint\limits_{Q_T\cap [\zeta<u_i<M]}\nabla f_i\cdot \nabla u_i +
\iint\limits_{Q_T}\tilde{u}_i f_i\log_\zeta^M(u_i)\\
\leq \int_\Omega H(u_i(0))+ \iint\limits_{Q_T\cap [\zeta<u_i<M]}\nabla f_i\cdot \nabla u_i + \frac{1}{2}\left(\iint\limits_{Q_T}\tilde u_i |f_i|^2
+\iint\limits_{Q_T}\tilde u_i\left|\log_\zeta ^M(u_i)\right|^2\right)\\
\leq \mathcal H(u_i(0))+ \iint\limits_{Q_T\cap [\zeta<u_i<M]}\nabla f_i\cdot \nabla u_i + \frac{1}{2}\iint\limits_{Q_T}\tilde u_i |f_i|^2+\frac{1}{2}\iint\limits_{Q_T}u_i\left|\log(u_i)\right|^2.
\end{multline*}
In the last line we also used $|\log_\zeta^M(z)|\leq |\log z|$ for small $\zeta$ and large $M$.
Controlling the subquadratic terms $\mathcal H(u_i(0)),\iint u_i|\log u_i|^2$ by $u_i(0)\in L^2(\Omega)$ and $\|u_i\|_{L^\infty(0,T;L^2)}\leq C_T$, and exploiting the dissipation estimate \eqref{eq:estimate_dissipation_leq_CT} we get
$$
0\leq C_T+\iint\limits_{Q_T\cap [\zeta<u_i<M]}\nabla f_i\cdot \nabla u_i
$$
for all $i=1\ldots N$.
This immediately passes to the limit $\zeta\searrow 0$, and recalling that by definition $f_i=m_i-(A\ub)_i$ we rewrite this limit as
$$
\iint\limits_{Q_T} \nabla (A\ub)_i\cdot\nabla u_i\leq C_T+\iint\limits_{Q_T}\nabla m_i\cdot \nabla u_i-\iint\limits_{[u_i\geq M]}\nabla f_i\cdot \nabla u_i.
$$
Summing over $i$'s, taking advantage of the coercivity $A\geq \lambda_A$, and suitably applying Young's inequality to the last two terms, this easily gives
$$
\iint\limits_{Q_T}|\nabla \ub|^2\leq C_T + C_A\sum\limits_i\iint\limits_{[u_i\geq M]}|\nabla f_i|^2
$$
and it is enough to show that the last term can be bounded uniformly in $M,\delta$.
To this end, observe by definition of $\tilde{u}_i=\min\{u_i,M\}$ that the dissipation estimate \eqref{eq:estimate_dissipation_leq_CT} immediately yields
$$
\iint\limits_{[u_i\geq M]}|\nabla f_i|^2 
= \iint\limits_{[u_i\geq M]}\frac{\tilde u_i}{M} |\nabla f_i|^2
\leq \frac{1}{M}\iint\limits_{Q_T}\tilde u_i|\nabla f_i|^2
\leq M^{-1}C_T\leq C_T
$$
if $M\geq 1$.
This finally gives the desired gradient estimate
\begin{equation*}
\label{eq:estimate_gradient_u}
\iint_{Q_T}|\nabla \ub|^2\leq C_T
\end{equation*}
uniformly in $M,\delta>0$.\\

In order to estimate the time derivative, let us recall that the weak formulation of \eqref{eq:zveri_apr} holds in duality with all $\mathcal C^1(\overline\Omega)$ (actually, even $H^1$) test functions.
Since $\|\ub\|_{L^{\infty}(0,T;(L^2)^N)}\leq C_T$ and we just proved that $\|\nabla \ub\|_{L^2(0,T;(L^2)^{N\times d})}\leq C_T$ we see that the products $\tilde u_i\nabla f_i$ and $\tilde u_i f_i$ are bounded respectively in $L^2(0,T;L^1)$ and $L^\infty(0,T;L^1)$ uniformly in $\delta,M$.
We conclude by duality in \eqref{eq:zveri_apr} that $\partial_t u_i=-\dive(\tilde u_i\nabla f_i)+\tilde u_i f_i+\delta\Delta u_i$ is bounded in $L^2(0,T;(\mathcal C^1(\overline\Omega)^*)^N)$.
To summarize, the solution $\ub=\lim \ub^{\eps_k}$ of \eqref{eq:zveri_apr} has the regularity
\begin{equation}
\label{eq:regularity_u_delta_M}
\|\ub'\|_{L^2(0,T;(\mathcal C^1(\overline\Omega)^*)^N)}+\|\ub\|_{L^\infty(0,T;(L^2)^N)}+\|\ub\|_{L^2(0,T;(H^1)^N)}\leq C_T
\end{equation}
uniformly in $M,\delta$.

\subsection{The limit $M\to\infty$}
 Here we want to pass to the limit $M\to +\infty$ in \eqref{E:apprp2} and \eqref{eq:zveri_apr} for fixed $\delta>0$.
 
 Define the limit operator $Q_\infty:(H^1)^N\to ((\mathcal{C}^1(\overline\Omega))^*)^N$ by
 $$
 \langle Q_\infty(\mathbf u), \mathbf w\rangle= \sum_{i=1}^N  \int_\Omega \left(u_i \nabla f_i\cdot \nabla w_i + u_i f_i w_i-\delta \nabla u_i \cdot \nabla w_i\right), \quad \forall \mathbf w \in \mathcal C^1(\overline\Omega)^N.
 $$
By \eqref{eq:regularity_u_delta_M}, the Banach-Alaoglu theorem and Aubin-Lions-Simon lemma, there exists a sequence $M_k \to +\infty$ ($\delta>0$ is fixed) such that, for the corresponding solutions $\ub^{M_k}$, we have
 $$
 \ub^{M_k}\rightharpoonup \ub\textrm{\ weakly\ in\ }L^2(0,T;(H^1)^N)\textrm{\ and\ weakly-*\ in\ } L^\infty(0,T;(L^2)^N),
 $$
$$
\ub^{M_k}\to \ub\textrm{\  strongly\ in\ } L^2(0,T;(L^2)^N)\textrm{\ and\ in\ } \mathcal C([0,T];((H^1)^*)^N),
$$
$$
(\ub^{M_k})^\prime\rightharpoonup \ub^\prime\textrm{\ weakly-*\ in\ }L^2(0,T;(\mathcal C^1(\overline\Omega)^*)^N)$$
with
\begin{equation}
\label{eq:estimate_Linfty-L2_L2-H1}
\|\ub '\|_{L^{2}(0,T;(\mathcal C^1(\overline\Omega)^*)^N)}	+	\|\ub\|_{L^\infty(0,T;(L^2)^N)}+\|\ub\|_{L^2(0,T;(H^1)^N)}\leq C_T 
\end{equation}
(uniformly in $\delta$).
Because the truncation operator $z\mapsto \tilde z=\min\{z,M\}$ is $1$-Lipschitz uniformly in $M$ and $u_i^{M_k}\to u_i$ strongly in $L^2$, one readily checks that
$$
\tilde u_i^{M_k}=\min\left\{u_i^{M_k},M_k\right\}\to u_i\qquad \mbox{in }L^2(0,T;L^2).
$$
Therefore the products pass to the strong-weak limit as before, $\tilde u^{M_k}_i\nabla f^{M_k}_i\rightharpoonup u_i\nabla f_i$ and $\tilde u^{M_k}_i f^{M_k}_i\rightharpoonup u_i f_i$ in $L^1(0,T;L^1)$, and $\ub$ solves the weak formulation
\begin{equation*}
\label{E:apprp3}
\mathbf u^\prime =Q_\infty(\mathbf u), \quad \mathbf u|_{t=0}=\mathbf u_0
\end{equation*}
of
\begin{equation}
\label{eq:zveri_apr3}
\left\{
\begin{array}{ll}
\displaystyle
\p_t u_i = -\operatorname{div} (u_i \nabla f_i) + u_i f_i+\delta \Delta u_i,&
x \in \Omega,
\\
u_i
\frac{
\partial f_i
}{
\partial \nu
}-\delta \frac{
\partial u_i
}{
\partial \nu
}
=
0
,
&
x \in \partial \Omega,
\\
u_i(0, x) = u_{0i}(x).
\end{array}
\right.
\end{equation}
In other words, $\ub$ solves
\begin{equation*}
 \label{eq:wsoldel}
\frac d {dt} \int_\Omega  \ub \cdot \mathbf w= \sum_{i=1}^N  \int_\Omega \left( u_i \nabla f_i\cdot \nabla w_i +  u_i  f_i w_i-\delta \nabla u_i \cdot \nabla w_i\right), \quad \forall \mathbf w \in (\mathcal{C}^1(\o\Omega))^N,
\end{equation*}
in the sense of scalar distributions.
Moreover, by the Lions-Magenes lemma \cite[Lemma 2.2.6]{ZV08} we see that $\ub\in \mathcal{C}_w([0,T]; (L^2)^N)$, hence the initial condition is taken in $\mathcal{C}_w([0,T]; (L^2)^N)$.
 
Before moving to the next limit $\delta\to 0$, we need to show that the dissipation estimate \eqref{eq:estimate_dissipation_leq_CT} also passes to the limit $M_k\to\infty$ .
This is not straightforward because of the cubic products $\tilde u_i(|\nabla f_i|^2+|f_i|^2)$, as $\tilde u_i^{M_k}$ does not a priori converge uniformly and $f_i^{M_k}$ should not converge strongly in $L^2(0,T;H^1)$.
In order to circumvent this technical difficulty we use a variant of the Banach-Alaoglu theorem in varying $L^2(\rd\mu^k)$ spaces:
\begin{lem}[compactness of vector-fields]
\label{lem:variant_banach_alaoglu_vector_fields}
Let $\mathcal O\subset \R^p$ be an open set, $\{\mu^k\}_{k\geq 0}$ a sequence of finite non-negative Radon measures narrowly converging to $\mu$, and $\mathbf v^k$ a sequence of vector fields on $\mathcal O$.
If
$$
\|\mathbf v^k\|_{L^2(\mathcal O,\rd\mu^k)}\leq C
$$
then there exists $\mathbf v\in L^2(\mathcal O,\rd\mu)$ such that, up to extraction of some subsequence,
$$
\forall \,\boldsymbol\zeta\in\mathcal C^\infty_c(\mathcal O):\qquad \lim\limits_{k\to\infty}\int_{\mathcal O}\mathbf v^k\cdot \boldsymbol\zeta\, \rd \mu^k=\int_{\mathcal O}\mathbf v\cdot \boldsymbol\zeta\, \rd \mu
$$
and
$$
\|\mathbf v\|_{L^2(\mathcal O,\rd\mu)}\leq \liminf\limits_{k\to\infty}\|\mathbf v^k\|_{L^2(\mathcal O,\rd\mu^k)}.
$$
\end{lem}

The proof of this fact by optimal transport techniques can be found in \cite{AGS06}; this lemma also follows from a variant of the Banach-Alaoglu theorem \cite[Proposition 5.3]{KMV15}.
We will apply this lemma component by component with $\mathcal O=(t_0,t_1)\times\Omega\subset \R^{d+1}$ for fixed $0\leq t_0\leq t_1\leq T$ and the sequence of measures $\rd\mu^k(t,x):=\tilde u_i^{M_k}(t,x)\rd x\rd t$, which converges narrowly to $\rd\mu(t,x)=u_i(t,x)\rd x\rd t$ due to the strong $L^2(Q_T)$ convergence $\tilde u^{M_k}_i\to u_i$.
The vector-fields of interest are of course the $\R^d$-valued $\nabla f_i^{M_k}$ and the scalar $f_i^{M_k}$.
Indeed, for fixed $0\leq t_0\leq t_1\leq T$ we have by \eqref{eq:estimate_dissipation_leq_CT} that
$$
\|\nabla f^{M_k}_i\|^2_{L^2(\mathcal O,\rd\mu^k)}=\int_{t_0}^{t_1}\int_{\Omega} |\nabla f^{M_k}_i|^2\tilde u_i^{M_k}	\leq \E(\ub(0))+C\delta T\leq C_T.
$$
Extracting a subsequence if needed, we see that there is a vector-field $\mathbf v_i\in L^2(\mathcal O,\rd\mu)$ such that
$$
\|\mathbf v_i\|^2_{L^2(\mathcal O,\rd\mu)}\leq \liminf\limits_{k\to\infty}\|\nabla f^{M_k}_i\|^2_{L^2(\mathcal O,\rd\mu^k)},
$$
and we claim that $\mathbf v_i=\nabla f_i$ in $L^2(\mathcal O,\rd\mu)$.
To see this, observe that from the weak convergence in Lemma~\ref{lem:variant_banach_alaoglu_vector_fields} there holds
$$
\lim\limits_{k\to\infty}\int_{t_0}^{t_1}\int_\Omega\nabla f_i^{M_k}\cdot \zeta \tilde u^{M_k}_i 
=\int_{t_0}^{t_1}\int_\Omega\mathbf v\cdot \zeta u_i 
$$
for all $\zeta\in \mathcal C^\infty_c((t_0,t_1)\times\Omega;\R^d)$.
On the other hand we already proved that $\tilde u_i^{M_k}\to u_i$ strongly in $L^2(0,T;L^2)$ and $\nabla f^{M_k}_i\rightharpoonup\nabla f_i$ weakly in $L^2(0,T;L^2)$, thus the product $\tilde u^{M_k}_i \nabla f^{M_k}_i\rightharpoonup u_i\nabla f_i$ weakly in $L^1(0,T;L^1)$.
Therefore
$$
\int_{t_0}^{t_1}\int_\Omega \nabla f_i\cdot \boldsymbol\zeta u_i	=\int_{t_0}^{t_1}\int_\Omega \mathbf v_i\cdot \boldsymbol\zeta u_i
$$
for all $\boldsymbol\zeta\in \mathcal C^\infty_c((t_0,t_1)\times\Omega;\R^d)$.
By density of $\mathcal C^\infty_c$ we conclude that $\mathbf v_i=\nabla f_i$ in $L^2(\mathcal O,\rd\mu)$, which shows in particular that the limit is independent of $t_0,t_1$ and the subsequence.
Whence
$$
\forall \,0\leq t_0\leq t_1:\qquad \int_{t_0}^{t_1}\int_{\Omega}u_i|\nabla f_i|^2
\leq \liminf\limits_{k\to\infty}
\int_{t_0}^{t_1}\int_{\Omega}\tilde u^{M_k}_i|\nabla f^{M_k}_i|^2.
$$
The argument is identical for the terms $\tilde u^{M_k}_i|f^{M_k}_i|^2$.
In order to finally retrieve the Entropy-Dissipation-Inequality, observe that $\ub^{M_k}\to\ub$ in $L^2(0,T;(L^2)^N)$ implies that $\E(\ub^{M_k}(t))\to \E(\ub(t))$ for almost every $t\in (0,T)$.
Taking the $\liminf\limits_{M_k\to\infty}$ in \eqref{eq:estimate_dissipation_leq_CT} we obtain
\begin{equation}
\label{eq:estimate_dissipation_leq_CT_M=infty}
\E(\ub(t_1))+ \sum\limits_i\int_{t_0}^{t_1}\int_\Omega u_i(|\nabla f_i|^2+|f_i|^2)
 \leq \E(\ub(t_0))+T \frac \delta {2  \lambda_A}\|\nabla \mb\|^2_{L^2}
\end{equation}
for almost every $0\leq t_0\leq t_1\leq T$.\\

\subsection{The limit $\delta\to 0$}
 
We are now  ready to pass to the last limit $\delta\to 0$. 
By \eqref{eq:estimate_Linfty-L2_L2-H1} we can find a sequence $\delta_k \to 0$ such that, for the corresponding solutions $\ub^{\delta_k}$ to \eqref{eq:zveri_apr3}, we have
$$
\ub^{\delta_k}\rightharpoonup \ub\textrm{\ weakly\ in\ }L^2(0,T;(H^1)^N)\textrm{\ and\ weakly-*\ in\ } L^\infty(0,T;(L^2)^N),
$$
$$
\ub^{\delta_k}\to \ub\textrm{\  strongly\ in\ } L^2(0,T;(L^2)^N)\textrm{\ and\ in\ } \mathcal C([0,T];((H^1)^*)^N),
$$
$$
(\ub^{\delta_k})^\prime\rightharpoonup \ub^\prime\textrm{\ weakly-*\ in\ }L^{2}(0,T;(\mathcal C^1(\overline\Omega)^*)^N)
$$
for every fixed $T>0$. 
Arguing as before for the products $u^{\delta_k}_i\nabla f^{\delta_k}_i$ and $u_i^{\delta_k} f^{\delta_k}_i$ we can take the strong-weak limits, and the $\delta\Delta u_i$ term in \eqref{eq:zveri_apr3} goes weakly to zero due to the $L^2(0,T;H^1)$ bound. Similarly to the previous step, $\ub\in \mathcal{C}_w([0,T]; (L^2)^N)$, and we can pass to the limit in the initial condition. 
Thus the limit $\ub$ is a weak solution to the original problem \eqref{eq:zveri}.
By standard diagonal extraction arguments it is easy to see that $\ub=\lim \ub^{\delta_k}$ can be chosen independent of the fixed time $T>0$. Thus the above convergence holds locally in time in $[0,\infty)$, the weak solution is global, and has the desired regularity in any finite time interval.

As for the Entropy-Dissipation-Inequality, we can repeat the exact same argument as in the previous section and pass to the $\liminf\limits_{\delta_k\to 0}$ in \eqref{eq:estimate_dissipation_leq_CT_M=infty} (with the last term $T \frac {\delta_k} {2  \lambda_A}\|\nabla \mb\|^2_{(L^2)^{N\times d}}$ vanishing in any finite time interval) to obtain \eqref{eq:EDI}.

\section{Long-time convergence}
\label{section:long_time_convergence}
This section is devoted to the proof of Theorem~\ref{T:lt1}, and without further mention we assume that for any $ I = \{i_1, \dots, i_r\} \subset \{1, \dots, N\}$, $i_1 < \dots 
< i_r$, and $j \notin I$ there holds
\begin{equation}
  \label{eq:am4}
    \begin{vmatrix}
      a_{i_1i_1} & \cdots  & a_{i_1i_r} & m_{i_1} \\
      \vdots   & \ddots & \vdots & \vdots \\
      a_{i_ri_1} & \cdots & a_{i_ri_r} & m_{i_r} \\
      a_{ji_1} & \cdots & a_{ji_r} & m_{j}
    \end{vmatrix}
   (x)\ge \kappa
\end{equation}
for some constant $\kappa\equiv \kappa(A,\mb)>0$.
\begin{rmk}
Letting $I = \emptyset$ in \eqref{eq:am4}, we see that necessarily $m_j(x) \geq \kappa>0$ for any $j$ and $x\in \Omega$, which means that there are only positive resources.
\end{rmk}
With this assumption, some elementary algebra shows in particular that the ideal free distribution
$$
\ub^\infty(x)=A^{-1}\mb(x)\geq c_\kappa>0
$$
becomes now a biologically relevant (non-negative) coexistence steady state of \eqref{eq:zveri} with $\fb\equiv 0$.
This particular distribution is clearly mathematically significant given the definition of the entropy \eqref{e:15}, and we shall prove below that it also attracts the long-time dynamics as in Theorem~\ref{T:lt1}.
However, \eqref{eq:am4} also implies the existence of a finite number of non-negative \emph{partial extinction} steady states constructed as follows:
given any $I \subset \{1, \dots, N\}$ and recalling that by definition $\fb=\mb-A\ub$, it is easy to see from \eqref{eq:am4} that the linear system
 \begin{equation*}
 \label{eq:vi}
 \left\{
 \begin{array}{l}
 u_i = 0 \quad (i \in I)
 \\
 f_j = 0 \quad (j \notin I),
 \end{array}
 \right.
 \end{equation*}
 has a unique solution $\mathbf{u}^I(x)$ satisfying $u_j(x)\geq c_{\kappa}>0$ for all $j\not\in I$ (thus $\ub^I(x)> 0$ componentwise).
  Note that those $\ub^I$ are trivially steady states of \eqref{eq:zveri} with $(u_i)_{i\in I}\equiv 0$ and $(f_j)_{j\not\in I}\equiv 0$, and that the ideal free distribution $\ub^\infty=\ub^{\emptyset}$ is the unique coexistence state obtained by taking $I=\emptyset\Leftrightarrow \fb=0$.
 In fact our condition \eqref{eq:am4} is equivalent to the hypothesis that all the components $(u^I_j)_{j\not\in J}$ of those steady states are positive and uniformly bounded away from zero, cf. \cite{KMV16-2} (by definition the other components $(u^I_i)_{i\in I}$ vanish identically). 
 
 The partial extinction set is then the collection of all such stationary solutions $\ub^I(x)$ for all possible choices of $I\subset \{1,\ldots,N\}$ with $I\neq \emptyset$, of which there is a finite combinatorial number $p_N$:
$$
\mathbf U^{ext}=\{\ub^{ext,1}(x),\ldots,\ub^{ext,p_N}(x)\}=\{\ub^I(x):\, \emptyset\neq I\subset \{1,\ldots,N\}\}\subset H^1(\Omega)^N.
$$ 
The \emph{critical entropy} $E^*>0$ appearing in Theorem~\ref{T:lt1} is then defined as the minimal entropy among all the partial extinction states,
\begin{equation} \label{crent}
E^*=\min\left\{\E(\mathbf u^{ext}):\quad \ub^{ext}\in \mathbf U^{ext}\right\},
\end{equation}
and depends only on $A,\mb$ (and $\Omega$). 

 Though biologically admissible, the partial extinction states are actually degenerate points in our analysis: whenever $u_i(x)\equiv 0$, the formal Riemannian structure from Section~\ref{section:gradient_flow} degenerates since the $i$-th tangent plane $T_{u_i}\Mm=\{0\}$ becomes trivial, see in particular the  definition of tangent norms \eqref{eq:def_tangent_norm} in terms of $\|.\|_{H^1(\rd u)}$ norms.
As a consequence we will need to stay away from those points.
This is particularly clear in the following functional inequality, which will be the key to proving the long-time convergence below and follows from a more general Poincar\'e-Beckner inequality established by us in \cite{KMV16-2}:
\begin{theo}
\label{th:kmpv2}
Let $\mathbf U \subset H^1(\Omega, \mathbb R^N)$ be a set of 
functions such that
\begin{enumerate}[(i)]
\item
$\mathbf u \ge 0$ for any $\mathbf u \in \mathbf U$;
\item
no sequence $\{\mathbf u^k\}_{k\geq 0} \subset \mathbf U$ converges strongly in $L^q(\Omega)^N$ to any of the partial extinction states $\ub^{ext}\in \mathbf U^{ext}$ for some $q\in [1,2)$.
\end{enumerate}
Then there exists a constant $C_{\mathbf U}>0$  such that
\begin{equation}
\label{eq:kmpv1}
\forall\,\mathbf u \in \mathbf U:\qquad \int_\Omega \sum_{i=1}^N |f_{i}|^2 \, \mathrm dx
\le
C_{\mathbf U}
\int_\Omega \sum_{i=1}^N u_{i} (|f_{i}|^2+ | \nabla f_{i} |^2) \, \mathrm dx
.
\end{equation}
\end{theo}
\begin{rmk} 
Apart from $\mathbf U$, the constant $C_{\mathbf U}$ also depends on  the upper bounds for $|\mathbf m|$ and $|A|$, on $\lambda_A$, and on $\kappa$ in \eqref{eq:am4}.
\end{rmk}

Condition (ii) means that the $\mathbf U$ must be 
separated from the finite set of partial extinction points $\mathbf U^{ext}$.
Moreover, this Poincar\'e-type inequality can be reinterpreted as an \emph{entropy-entropy production} inequality, as is common in the framework of gradient flows in Wasserstein spaces. 
Indeed from Section~\ref{section:gradient_flow} and the formal Riemannian structure the right-hand side is nothing but the dissipation $\mathcal D(\ub) =\|\grad_\D \mathcal E(\ub)\|^2_{T_\ub\Mm}$, and recalling $\E(\ub)\lesssim \|\fb \|^2_{L^2(\Omega)}$ from the coercivity, the left hand side controls $\E(\ub)=\E(\ub)-\E(\ub^\infty)$.
Thus \eqref{eq:kmpv1} gives the entropy-entropy dissipation control $\mathcal D(\ub)\geq C (\E(\ub)-\E(\ub^\infty))$, which classically implies convergence in the entropy sense.\\

We are now in position of proving the long-time convergence:
\begin{proof}[Proof of Theorem \ref{T:lt1}] Recalling the definition of the critical entropy \eqref{crent} and given a subcritical initial data
$$
\E(\ub^0)<E^*,
$$
we introduce the set
$$
\mathbf U:=\left\{\ub=(u_1,\ldots,u_N):\qquad
u_i\geq 0,\quad u_i\in H^1(\Omega),\quad \mathcal \E(\ub)\leq \E( \ub^0)
\right\}
$$
(depending only on $\ub^0$, $A$, $\mb$, and $\Omega$).
From the EDI \eqref{eq:EDI} we see that we have invariance
$$
\ub(0)=\ub^0\in \mathbf U \quad \Rightarrow \quad \ub(t)\in \mathbf U \text{ for a.a. }t\geq 0
$$
along the time-evolution, and we claim that $\mathbf U$ meets the assumptions of Theorem~\ref{th:kmpv2}.
To see this, assume by contradiction that there is a sequence $\ub^k\in \mathbf U$ such that $\ub^k\to \ub^{ext}$ strongly in $L^q(\Omega)$ for some partial extinction state $\ub^{ext}\in \mathbf U^{ext}$ and some $q\in [1,2)$.
Since $\frac{\lambda_A}{2}\|\ub^k-\ub^{\infty}\|^2_{L^2(\Omega)}\leq \E(\ub^k)\leq \E(\ub^0)$ we see that $\{\ub^k\}$ is bounded in $L^2(\Omega)$, and up to extraction of a subsequence we can therefore assume that $\ub^k\rightharpoonup \ub$ weakly in $L^2$ for some limit $\ub$.
By uniqueness of the limit we see that $\ub=\ub^{ext}$, and by lower semi-continuity
$$
\E(\ub^{ext})\leq \liminf\limits_{k\to\infty} \E(\ub^k)\leq \E(\ub^0)<E^*.
$$
This is impossible by the definition \eqref{crent} of $E^*$, which entails the claim.

From \eqref{e:15} and the coercivity $A\geq \lambda_A$ we recall that $\E(\ub)\leq C_A\int_{\Omega}|\fb|^2$.
We can therefore apply Theorem~\ref{th:kmpv2} in the Entropy-Dissipation-Inequality \eqref{eq:EDI} and conclude that there is $\gamma=\gamma_{\mathbf U}>0$, depending only on $\E(\ub^0)$ and the data, such that for a.e. $t\geq t_0\geq 0$ there holds
$$
\E(\ub(t))+\gamma\int_{t_0}^t\E(\ub(s))\rd s\leq \E(\ub(t))+\sum\limits_i\int_{t_0}^t\int_\Omega (|\nabla f_i|^2+|f_i|^2)u_i\leq \E(\ub(t_0)).
$$
Hence $t \mapsto \E(\ub(t))+\gamma\int_{0}^t\E(\ub(s))\rd s$ is monotone non-increasing, and therefore
$$
\frac{d\E}{dt}+\gamma\E\leq 0
$$
in the sense of scalar distributions $\mathcal D'(0,\infty)$.
This immediately implies the exponential decay \eqref{eq:din} for a.e $t$ by a standard Gr\"onwall argument. Finally, since $\ub\in \mathcal{C}_w([0,\infty); L^2(\Omega)^N)$, the function $t \mapsto \E(\ub(t))$ is lower semicontinuous, and \eqref{eq:din} extends to all $t\geq 0$.
\end{proof}
\begin{rmk}
From the biological perspective, the distributions $\ub^I$ ($I\neq \emptyset$) describe scenarios when some of the species $(u_i)_{i\in I}$ have died out, and the survivors $(u_j)_{j\not\in J}$ compose a (lower-dimensional) ideal free distribution. It is important to point out that these partial ideal free distributions (which we need to avoid to secure the entropy-entropy production inequality) are unstable and repulsive: if a small $L^\infty$ density of any of the extinct populations $(u_i)_{i\in I}$ is reintroduced, its fitness $f_i$ will be positive and bounded away from zero (see \cite{KMV16-2}), therefore the environment is favorable to that species and it will unlikely go extinct again. That is why we conjecture that \eqref{eq:din} holds for any $\ub^0\geq 0$ (unless some component of $\ub^0$ is identically zero).\end{rmk}

\section*{Acknowledgments}
SK and DV were partially supported by CMUC -- UID/MAT/00324/2013, funded by the Portuguese Government through FCT/MCTES and co-funded by the European Regional Development Fund through the Partnership Agreement PT2020.
LM was supported by the Portuguese National Science Foundation through through FCT fellowship SFRH/BPD/ 88207/2012 and by the UT Austin/Portugal CoLab project \emph{Phase Transitions and Free Boundary Problems}. DV was also supported by the Russian Science Foundation (14-21-00066 -- Voronezh State University). 
\section*{Appendix} 

Assume that $X\subset Y$ are two Hilbert spaces, with
continuous embedding operator $i:X\to Y$, and that $i(X)$ is dense in
$Y$.
The adjoint operator $i^*:Y^*\to X^*$ is continuous and,
since $i(X)$ is dense in $Y$, one-to-one. Since $i$ is one-to-one,
$i^*(Y^*)$ is dense in $X^*$, and one may identify $Y^*$ with a
dense subspace of $X^*$. Due to the Riesz representation theorem,
one may also identify $Y$ with $Y^*$. We arrive at the chain of
inclusions:
\begin{equation}
X\subset Y\equiv Y^*\subset X^*
\end{equation}
with dense and continuous embeddings.
Observe that in
this situation, for $f\in Y, u\in X$, their scalar product in $Y$
coincides with the $<X^*,X>$ duality
\begin{equation}
\label{ttt} (f,u)_Y=\langle f,u \rangle_{X^*,X}.
\end{equation}
Such triples $(X,Y, X^*)$ are called \emph{Hilbert triples} (sometimes also referred to as Gelfand or Lions triples), see, e.g., \cite{Te79,ZV08} for more details. 

\begin{lem} \label{l:ap} Let $$X\subset Y\subset X^*$$ be a Hilbert triple. Let $\mathcal{A}:X\to X^*$ be a linear continuous operator such that $$\langle \mathcal A u , u \rangle \geq \alpha \|u\|_X^2$$ for all $u\in X$ and some common $\alpha>0$. Let $V$ be a Banach space such that $$X\subset V\subset Y$$ where the first embedding is compact and the second is continuous. Assume that both $X$ and $V$ are separable. Let $$\mathcal{Q}: V\to X^*$$ be a continuous operator. Assume that \be\label{e:q1}\|\mathcal{Q}(u)\|_{X^*} \leq C (1+\|u\|_V)\ee for all $u \in X$. Then the Cauchy problem \begin{equation}\label{E:acp1} u^\prime(t) +\mathcal Au(t)=\mathcal Q(u(t)), \quad  u|_{t=0}=u_0,
\end{equation} has a solution in the class \be L^{2}(0,T;X)\cap H^{1}(0,T;X^*)\cap C([0,T];Y)\ee for every $u_0\in Y$. \end{lem}
\bibliographystyle{plain}

We omit the proof since a more general statement is proven in  \cite{SV16}, cf. also  \cite[Section 6.3]{ZV08}, \cite[Section 4]{V9}. 

\end{document}